\documentclass[12]{amsart}
\usepackage{amssymb}
\usepackage{amsfonts}
\usepackage{latexsym}   
\usepackage{amssymb}    
\usepackage{amsmath}    
\usepackage{amsbsy}
\usepackage{amsthm}
\usepackage{amsgen}
\usepackage{amsfonts}
\usepackage{array}
\usepackage{epsfig}
\usepackage{hyperref}
\usepackage{psfrag}
\usepackage[all]{xy}
\usepackage{color}

\theoremstyle{plain}

\newtheorem{algorithm}{Algorithm}[section]

\newtheorem{corollary}[algorithm]{Corollary}

\newtheorem{definition}[algorithm]{Definition}

\newtheorem{lemma}[algorithm]{Lemma}

\newtheorem{theorem} [algorithm] {Theorem}

\newtheorem{proposition}[algorithm]{Proposition}
\newtheorem{remark}[algorithm]{Remark}
\newtheorem*{Remark}{Remark}

\numberwithin{equation}{algorithm}

\newtheorem*{Main}{Main Theorem}

\newtheorem{DSthm}[algorithm]{Double Soul Theorem}

\def\lpq{L_{p, q}}

\def\rrr{\mathbb{R}}
\def\ccc{\mathbb{C}}
\def\zzz{\mathbb{Z}}

\DeclareMathOperator{\Isom}{Isom}
\DeclareMathOperator{\Fix}{Fix}

\def\bdm{\begin{displaymath}}
\def\edm{\end{displaymath}}
\def\beq{\begin{equation}}
\def\eeq{\end{equation}}
\def\bes{\begin{equation*}}
\def\ees{\end{equation*}}
\def\epcm{\end{picture}\end{center}\end{minipage}}
\def\bpcm{\begin{minipage}{80pt}\begin{center}\begin{picture}}

\def\t2{T^2}

\def\f4{F_4}
\def\g2{G_2}

\def\p2{\frac{\pi}{2}}

\def\Fix{\textrm{Fix}}
\def\rk{\textrm{rk}}

\def\dim{\textrm{dim}}

 \numberwithin{equation}{section}
  \numberwithin{figure}{section}

\newtheorem*{ack}{Acknowledgements}

\begin{document}
\newcommand{\comment}[1]{\vspace{5 mm}\par \noindent
\marginpar{\textsc{Note}}
\framebox{\begin{minipage}[c]{0.95 \textwidth}
#1 \end{minipage}}\vspace{5 mm}\par}

\title[Non-negatively curved $6$-manifolds with almost maximal symmetry rank]{Non-negatively curved $6$-manifolds with almost maximal symmetry rank }

\author[Escher]{Christine Escher}
\address[Escher]{Department of Mathematics, Oregon State University, Corvallis, Oregon}
\email{tine@math.orst.edu}

\author[Searle]{Catherine Searle}
\address[Searle]{Department of Mathematics, Statistics, and Physics, Wichita State University, Wichita, Kansas}
\email{searle@math.wichita.edu}


\date{\today}

\begin{abstract}

We classify closed, simply-connected, non-negatively curved $6$-manifolds of almost maximal symmetry rank up to equivariant diffeomorphism.

\end{abstract}
\maketitle

\section{Introduction}

For the class of closed, simply-connected Riemannian manifolds there are no known obstructions that allow us to distinguish between positive and non-negative sectional curvature, in spite of the fact that the number of known examples of manifolds of non-negative sectional curvature is vastly larger than those known to admit a metric of positive sectional curvature.

The introduction of symmetries, however,  allows us to distinguish between such classes. An important first case to understand is that of maximal symmetry rank, where the symmetry rank is defined to be the rank of the isometry group:
$\text{sym}\rk(M^n)=\rk(\Isom(M^n)).$
For manifolds of strictly positive sectional curvature, a classification up to equivariant diffeomorphism was obtained by Grove and Searle \cite{GS}. They showed that for such manifolds, the maximal symmetry rank is equal to $\lfloor (n+1)/2 \rfloor$. For closed, simply-connected manifolds of non-negative sectional curvature, the maximal symmetry rank is conjectured to be $\lfloor 2n/3\rfloor$ (see Galaz-Garc\'ia and Searle \cite{GGS1} and Escher and Searle \cite{ES}). A classification for the latter has been obtained, but only in dimensions less than or equal to nine (see \cite{GGS1} and Galaz-Garc\'ia and Kerin \cite{GGK}, for dimensions less than or equal to $6$ and \cite{ES} for dimensions $7$ through $9$) and the upper bound for the symmetry rank has been verified for dimensions less than or equal to $12$ (see \cite{GGS1} and \cite{ES}).

A natural next step is  the case of almost maximal symmetry rank. In positive curvature, a homeomorphism classification was obtained by Rong \cite{Ro}, in dimension $5$,  and Fang and Rong \cite{FR}, for dimensions greater than or equal to $8$, using work of Wilking \cite{Wi1}. In non-negative curvature, a  homeomorphism classification was obtained independently by Kleiner \cite{K} and Searle and Yang \cite{SY}, in dimension $4$. This classification was later improved to equivariant diffeomorphism by Grove and Wilking \cite{GW}. A diffeomorphism classification in dimension $5$ was obtained by Galaz-Garc\'ia and Searle \cite{GGS2}. 

In this article we consider closed, simply-connected Riemannian $6$-manifolds admitting a metric of non-negative sectional curvature and an effective, isometric
torus action of almost maximal symmetry rank and prove the following classification theorem.

\begin{Main} Let $T^3$ act isometrically and effectively on $M^6$, a closed, simply-connected, non-negatively curved Riemannian manifold. Then $M^6$ is equivariantly diffeomorphic to $S^3\times S^3$ or a torus manifold.
\end{Main}

  Closed, orientable manifolds of dimension $2n$ admitting a smooth $T^n$-action with non-empty fixed point set, are called {\em torus manifolds}. 
Non-negatively curved torus manifolds were classified up to equivariant diffeomorphism by Wiemeler \cite{Wie} (see Theorem \ref{Wiemeler}). 
 In dimension $6$, they are equivariantly diffeomorphic to $S^6$, $\ccc P^3=S^7/T^1$, or the quotient by (1), a free linear circle action on $S^3\times S^4$, (2), a free linear $T^2$-action on 
 $S^3\times S^5$ or (3), a free linear $T^3$-action on $S^3\times S^3\times S^3$. In the process of classifying bi-quotients of dimension $6$, De Vito \cite{DV} has given a  classification of these manifolds up to diffeomorphism.
   It is worth noting that Kuroki \cite{Ku}, using torus graphs, has obtained an Orlik-Raymond type classification of $6$-dimensional torus manifolds with vanishing odd degree cohomology without curvature restrictions.
   
 \begin{Remark} In all cases except one, $M^6$ is equivariantly diffeomorphic to one of the above manifolds with a linear torus action. In the case where $M^6$ admits only rank one isotropy and the $T^3$-action is $S^1$-fixed point homogeneous then we can only say that $M^6$ is equivariantly diffeomorphic to $S^3\times S^3$ with a smooth $T^3$-action.
  \end{Remark}

 The paper is organized as follows. Section \ref{2} covers preliminary material required for the proof of the Main Theorem. Section \ref{3} contains two general results about manifolds that decompose 
 as disk bundles without curvature restrictions. In Section \ref{4}, we prove the Main Theorem.
\begin{ack} 
Both authors would like to thank Michael Wiemeler and a referee for pointing out an omission in a previous version of the paper. We are also indebted to the same referee for many helpful comments and suggestions.
This material is based in part upon work supported by the National Science Foundation under Grant No. DMS-1440140 while both authors were in residence at the Mathematical Sciences Research Institute in Berkeley, California, during the Spring 2016 semester. Catherine Searle would also like to acknowledge support  by grants from the National Science Foundation (\#DMS-1611780), as well as  from the Simons Foundation (\#355508, C. Searle).
 
\end{ack}
\section{Preliminaries}\label{2}

In this section we will gather basic results and facts about transformation groups,  the topological classification of six dimensional manifolds and $G$-invariant manifolds of non-negative curvature.

\subsection{Transformation Groups}\label{s2}

Let $G$ be a compact Lie group acting on a smooth manifold $M$. We denote by $G_x=\{\, g\in G : gx=x\, \}$ the \emph{isotropy group} at $x\in M$ and by $G(x)=\{\, gx : g\in G\, \}\simeq G/G_x$ the \emph{orbit} of $x$. Orbits are called {\em principal}, {\em exceptional} or {\em singular}, depending on the relative size of their isotropy subgroups; that is, principal orbits correspond to those orbits with the smallest possible isotropy subgroup, an orbit is called exceptional when its isotropy subgroup is a finite extension of the principal isotropy subgroup and singular when its isotropy subgroup is of strictly larger dimension than that of the principal isotropy subgroup.  

The \emph{ineffective kernel} of the action is the subgroup $K=\cap_{x\in M}G_x$. We say that $G$ acts \emph{effectively} on $M$ if $K$ is trivial. The action is called \emph{almost effective} if $K$ is finite. 

 We will sometimes denote the \emph{fixed point set} $M^G=\{\, x\in M : gx=x, g\in G \, \}$ of the $G$-action by $\Fix(M ; G )$.  One measurement for the size of a transformation group $G\times M\rightarrow M$ is the dimension of its orbit space $M/G$, also called the {\it cohomogeneity} of the action. This dimension is clearly constrained by the dimension of the fixed point set $M^G$  of $G$ in $M$. In fact, $\dim (M/G)\geq \dim(M^G) +1$ for any non-trivial action with fixed points. In light of this, the {\it fixed-point cohomogeneity} of an action, denoted by $\textrm{cohomfix}(M;G)$, is defined by
\[
\textrm{cohomfix}(M; G) = \dim(M/G) - \dim(M^G) -1\geq 0.
\]
A manifold with fixed-point cohomogeneity $0$ is also called a {\it $G$-fixed point homogeneous manifold}.

\subsection{Topological Classification of $6$-manifolds}

Note that throughout the paper we will use the convention that all homology groups have integer coefficients, unless otherwise specified.

The topological classification of simply-connected, closed, oriented $6$-manifolds has been completed in a sequence of articles by C.T.C. Wall \cite{Wa}, P. Jupp \cite{J}, and 
A. \v{Z}ubr \cite{Z1, Z2, Z3}.   We will focus on the classification of closed, simply-connected, oriented $6$-manifolds with torsion free homology.  The classification theorem below is due to C. T. C. Wall in the case of smooth spin manifolds, \cite{Wa}, and in the final form due to P. Jupp \cite{J}.  
We first describe the basic invariants used to classify $6$-dimensional, closed,  simply-connected, oriented, smooth manifolds,  $M$, with torsion free homology \cite{J}.

\begin{theorem}\cite{J}  Let $M$ be a $6$-dimensional, closed,  simply-connected, oriented, smooth manifold with torsion free homology. The basic invariants 
used to classify $M$ are as enumerated below.
\begin{enumerate}
\item $H:= H^2(M)$, a finitely generated free abelian group;
\item  $b:= b_3(M) = \text{rk}_\zzz(H^3(M)) \in 2\zzz$ since $H^3(M)$ admits a non-degenerate symplectic form;
\item $F:=F_M: H^2(M) \otimes H^2(M)  \otimes H^2(M)  \longrightarrow \zzz$ a symmetric trilinear form given by the cup product evaluated on the orientation class;
\item $p:= p_1(M) \in H^4(M)$, the first Pontrjagin class;
\item $w:=w_2(M) \in H^2(M;\zzz_2)$, the second Stiefel-Whitney class.
\end{enumerate}
\end{theorem}
We now use Poincar\'{e} duality to identify $H^4(M)$ with $\text{Hom}_\zzz(H^2(M);\zzz)$ so that $p_1(M)$ can be interpreted as a linear form on $H^2(M)$ and we let $x \cdot y \cdot z$ denote $F_M(x \otimes y \otimes z)$ for $x,y,z \in H^2(M)$.  
\begin{definition}{\bf (Admissibility)}
 
The system of invariants $(H, b, w, F, p)$ is called {\em admissible} if and only if for every $\omega \in H$ and $T \in H^*:=\text{Hom}_\zzz(H;\zzz)$ with $\rho_2(\omega)=w$ and $\rho_2(T) = 0$ where $\rho_2: \zzz \longrightarrow \zzz_2$ is reduction modulo $2$, the following congruence holds:
$$\omega^3 \equiv (p + 24 T)\, \omega \mod 48.$$
\end{definition}
\begin{definition}{\bf (Equivalence)}
 Two systems $(H, b, w, F, p)$ and $(H', b', w', F', p')$ are called {\em equivalent} if and only if $b = b'$ and there exists an isomorphism $\alpha: H \longrightarrow H'$ such that $\alpha(w) = w', \alpha^*(F') = F,\alpha^*(p') = p.$
\end{definition}
We are now ready to state the classification result:
 \begin{theorem}\cite{J}  The assignment 
 $$M \mapsto (\frac{b}{2}, H^2(M), w_2(M), F_M, p_1(M))$$
 induces a 1-1 correspondence between oriented diffeomorphism classes of simply-connected, closed, oriented,  $6$-dimensional, smooth manifolds
 with torsion free homology, and equivalence classes of admissible systems of invariants.  
 \end{theorem} 

Note that A. \v{Z}ubr generalized Wall's theorem in a different direction:  he proved a classification theorem for simply-connected, smooth spin 
manifolds with not necessarily torsion free homology \cite{Z1}, and then in \cite{Z2, Z3}  also obtained Jupp's theorem and proved that algebraic isomorphisms of systems 
of invariants can always be realized by orientation preserving diffeomorphisms.  

Observe that the first invariant $\frac{b}{2}$ is completely independent of the other invariants which implies that the following splitting theorem 
holds.
 \begin{corollary} \cite{Wa} Every simply-connected, closed, oriented, $6$-dimensional, smooth manifold $M$ 
 admits a splitting $M = M_0 \, \sharp \, \frac{b}{2} (S^3 \times S^3)$ as a connected sum of a core $M_0$ with $b= b_3(M_0) = 0$ and 
 $\frac{b}{2}$ copies of $S^3 \times S^3$. 
  \end{corollary}

 The following corollary is an immediate consequence.
 \begin{corollary} \label{Wall}
 Let $M$ be a simply-connected, closed, oriented, $6$-dimensional, smooth manifold with 
 $$H_i(M^6)\cong H_i(S^3 \times S^3) \,\text{for all} \,\,i.$$
 Then $M^6$ is diffeomorphic to $S^3\times S^3$.
 \end{corollary}

\subsection{$G$-manifolds with non-negative curvature}

We now recall some general results about $G$-manifolds with non-negative curvature which we will use throughout.
Recall that fixed point homogeneous manifolds of positive curvature were classified in \cite{GS}.  More recently, 
Spindeler  \cite{Spi} proved the following theorem which characterizes non-negatively curved $G$-fixed point homogeneous manifolds.

\begin{theorem}\label{Spindeler} \cite{Spi} Assume that $G$ acts fixed point homogeneously  on a complete non-negatively curved Riemannian manifold $M$. Let $F$ be a fixed point component of maximal dimension. Then there exists a smooth submanifold $N$ of $M$, without boundary, such that $M$ is diffeomorphic to the normal disk bundles $D(F)$ and $D(N)$ of $F$ and $N$ glued together along their common boundaries;
\bdm
M  = D(F) \cup_{\partial} D(N).
\edm
Further, $N$ is $G$-invariant and contains all singularities of $M$ up to $F$.
\end{theorem}

Let $\Isom_F(M)$ be the subgroup of the isometry group of $M$ that leaves $F$ invariant. 
The following lemmas from \cite{Spi} will also be important.

\begin{lemma}\cite{Spi}\label{equivariance} Let $M$ be a non-negatively curved fixed point homogeneous $G$-manifold, with $M$, $F$ and $N$ as in Theorem \ref{Spindeler} and $K=\Isom_F(M)$. Then there exists an $K$-equivariant diffeomorphism $b: \partial D(N)\rightarrow \partial D(F)$ and $M$ is $K$-equivariantly diffeomorphic to 
$D(F) \cup_{\partial} D(N).$
\end{lemma}

\begin{lemma}\cite{Spi}\label{codim} Let $M$ and $N$ be as in Theorem \ref{Spindeler} and assume that $\pi_1(M)=0$ and $G$ is connected. Then $N$ has codimension  greater than or equal to $2$ in $M$.
\end{lemma}

The next theorem from Galaz-Garc\'ia and Spindeler \cite{GGSp} covers the special case when both $F$ and $N$ are fixed point sets of the $G$-action and generalizes the Double Soul Theorem  for $S^1$-fixed point homogeneous actions of \cite{SY}.
\begin{DSthm} \cite{SY}\cite{GGSp}\label{DST}
Let $M$ be a non-negatively curved $G$-fixed point homogeneous Riemannian manifold, where the principal isotropy group of the $G$ action is $H$. If $\Fix(M, G)$ contains at least two connected components $F$  and $N$ with maximal dimension, one of which is compact, then $F$ and $N$ are isometric and $M$ is diffeomorphic to an $S^{k+1}$-bundle over $F$, where $S^k=G/H$.
\end{DSthm}

 Since fixed point homogeneous manifolds with either positive or non-negative lower curvature bounds decompose as unions of disk bundles, the following purely topological lemma from \cite{ES} will 
 be useful. 
\begin{lemma}\label{l:pi1} \cite{ES} Let $M$ be a  manifold with $\rk(H_1(M))=k$, $k\in\zzz^+$. If $M$ admits a disk bundle decomposition 
\bdm
M=D(N_1)\cup_{E} D(N_2), 
\edm
where $N_1$, $N_2$ are smooth submanifolds of $M$ and $N_1$ is orientable and of codimension two, then both 
$\rk(H_1(N_1))$ and $\rk(H_1(N_2))$ are less than or equal to $k+1$.
\end{lemma}

An important subclass of manifolds admitting an effective torus action are the so-called {\it torus manifolds}. 

\begin{definition}[{\bf Torus Manifold}] A {\em torus manifold} $M$ is a $2n$-dimensional closed, connected, orientable, smooth manifold with an effective smooth action of 
an $n$-dimensional torus $T$ such that $M^T\neq \emptyset$.
\end{definition}
A related concept is that of an  {\em isotropy-maximal} $T^k$-action on $M^n$, when there exists a point in $M$ whose isotropy group is maximal, namely of dimension $n-k$ (see Ishida \cite{I}, see also \cite{ES}).
Note that a torus manifold, $M^{2n}$, is an example of a manifold admitting an isotropy-maximal $T^n$-action. In fact, we may characterize torus manifolds as $2n$-dimensional closed, connected, orientable, smooth manifolds with an effective and isotropy-maximal smooth $T^n$-action.

The following important theorem from \cite{Wie} gives a classification up to equivariant diffeomorphism of non-negatively curved torus manifolds.

\begin{theorem}\label{Wiemeler}   \cite{Wie} Let $M^{2n}$ be a simply-connected, non-negatively curved Riemannian manifold admitting an isometric, effective, and isotropy-maximal $T^n$-action. Then $M$ is equivariantly diffeomorphic to 
a quotient of a free linear torus action of 
\bdm
\mathcal{Z}=\prod_{i<r} S^{2n_i} \times \prod_{i\geq r} S^{2n_i-1},\, \, n_i\geq 2.
\edm
\end{theorem}

We now recall Theorem A from \cite{ES}, which generalizes Theorem \ref{Wiemeler} and is used in the proof of the Main Theorem.

\begin{theorem}\label{ES}\cite{ES}
 Let $M^n$, a closed, simply-connected, non-negatively curved Riemannian manifold admitting an isometric, effective and isotropy-maximal $T^k$-action, where $k\geq \lfloor (n+1)/2\rfloor$. Then $M$ is equivariantly diffeomorphic to the free linear quotient of $\mathcal{Z}$, 
$$\mathcal{Z}=\prod_{i<r} S^{2n_i} \times \prod_{i\geq r} S^{2n_i-1},\, \, n_i\geq 2,$$ a product of spheres of dimensions greater than or equal to $3$ and with $n\leq \dim(\mathcal{Z})\leq 3n-3k$.  
\end{theorem}

\section{Disk Bundle Decompositions}\label{3}
In this section we present two general topological results about manifolds which decompose as unions of disk bundles. Note that these results are curvature independent.

The first theorem allows us to identify the fundamental group of $E$ in the disk bundle decomposition.
 
 \begin{theorem}\label{VanKampen} Let $M^n$ be a simply-connected manifold that decomposes as the union of two disk bundles as follows:
$$M^n=D^k(N_1)\cup_E D^l(N_2).$$
Then the following hold:
\begin{enumerate}
\item  If $k=l=2$ and 
$\pi_2(N_i)=0$ for $i=1, 2$ and $\pi_ 1(N_1)$ is not a finite cyclic group, then  $\pi_1(E) \cong \zzz^2$.
\item If $k\geq 3$, then $\pi_1(E) \cong \pi_1(N_1)$.
\end{enumerate}
\end{theorem}
\begin{proof}
Case (1):  
Assume that $k=l=2$.  Then $E$ is a circle bundle over $N_1$ and also over $N_2$, where $\pi_2(N_1)=\pi_2(N_2)=0$ .  Hence we obtain the following short exact sequences from the long exact sequences in homotopy:

 $$0 \longrightarrow \pi_1(S^1_j) \overset{i_j*} \longrightarrow  \pi_1(E) \overset{f_j^*}  \longrightarrow  \pi_1(N_j)  \longrightarrow 0\,, \textrm{ for } j\in \{1, 2\}$$

  \vspace{0.2cm}
  
 Now let $U_1 = i_1^*(\pi_1(S^1_1))$ and  $U_2 = i_2^*(\pi_1(S^1_2))$.  Then $\pi_1(N_1) = \pi_1(E)/U_1$ and $\pi_1(N_2) = \pi_1(E)/U_2$,
 so we get the following commutative diagram:
 
   \begin{equation*} \tag{$\ast$}
\begin{split}
\xymatrix{
&  \pi_1(E)    \ar[r]^{f_2^*}   \ar[d]^{f_1^*}  &  \pi_1(N_2)   \ar[d]^{}    &    \\
 &  \pi_1(N_1)  \ \ar[r]^{}  & \pi_1(E)/U_1 U_2  }
\end{split}
\end{equation*}

Here the lower map is given by 
$$\pi_1(N_1) \cong \pi_1(E)/U_1 \longrightarrow \pi_1(E)/U_1 U_2 \,,$$
and the same is true for $\pi_1(N_2)$.  Now by Seifert-van Kampen (universal property), there exists a morphism $h: \pi_1(M) \longrightarrow \pi_1(E)/U_1 U_2 $
making the following diagram commute:  
 
   \begin{equation*}
\begin{split}
\xymatrix{
&  \pi_1(E)    \ar[r]^{f_2^*}   \ar[d]^{f_1^*}  &  \pi_1(N_2)   \ar@/^2pc/[ddr]_{}  \  \ar[d]^{}    &  {} &  \\
 &  \pi_1(N_1)   \ar@/_2pc/[drr]_{} \ \ar[r]^{}  &    \pi_1(M)    \ar[rd]^{h}   & {} \\
 &  {}   &  {} & \pi_1(E)/U_1 U_2 }
\end{split}
\end{equation*}
 
 Since all the maps in ($\ast$) are surjective, $h$ must be surjective.  But since $\pi_1(M) = 0$, this implies that $\pi_1(E) \cong U_1 U_2 $.  
 Note that both $U_1$ and $U_2$ are normal in $\pi_1(E)$.  If in addition $U_1 \cap U_2 = \{1\}$, then $\pi_1(E) \cong U_1 \times U_2 
 \cong \zzz^2$ and the theorem follows.  If $U_1 \cap \, U_2  \ne \{1\}$, then $\pi_1(N_1) \cong U_1U_2/U_1 \cong U_2/{U_1 \cap U_2}$.
 But $U_1 \cap U_2$ is a normal subgroup of $U_2 \cong \zzz$, hence $U_1 \cap U_2 \cong n \, \zzz$ for some $n \in \zzz$.  It follows 
 that $\pi_1(N_1) \cong U_2/{U_1 \cap U_2} \cong \zzz/n \zzz $ which is a contradiction to the hypothesis that $\pi_1(N_1)$ is not finite cyclic.
 Hence $U_1 \cap U_2 = \{1\}$ and  $\pi_1(E) \cong U_1 \times U_2  \cong \zzz^2$.
 
 Case (2):  Assume now that $k \ge 3$.   Then $E$ is a $S^{k-1}$ bundle over $N_1$ and hence by the long exact sequence in homotopy 
 $\pi_1(E) \cong \pi_1(N_1)$.   
 
 \end{proof}

For manifolds of dimension greater than or equal to $3$ that decompose as a union of disk bundles, the following general theorem allows us to impose strong restrictions on the fundamental groups of the bases of these bundles.

\begin{theorem}\label{cyclic} Let $M^n$ be a closed, simply-connected manifold admitting a smooth $S^1$-fixed point homogeneous action, with $n\geq 3$. Let $F^{n-2}$ be the unique fixed point component of $S^1$ of codimension two. Suppose further that $M$ decomposes as a union of disk bundles over $F$ and over $N$, that is, 
$$M=D(F)\cup_{E} D(N),$$  where $N$ is a codimension two submanifold invariant under the $S^1$-action. Suppose further that all singularities of the $S^1$-action outside of $F$ are contained in $N$.
Then $\pi_1(F)$ and $\pi_1(N)$ are cyclic groups.
\end{theorem}

Note that this theorem is purely topological in nature and generalizes a similar result for closed, simply-connected non-negatively curved  $S^1$-fixed point homogeneous $5$-manifolds (see the proof of Proposition 3.6 in  \cite{GGSp}). Indeed, the arguments in their proof are, for the most part, completely general. We will therefore only give an outline of the proof of Theorem \ref{cyclic} here, as the reader may refer to \cite{GGSp} for more details.
\begin{proof}

Let $E := \partial D(F) \cong \partial D(N)$. Then, as shown in the proofs of Propositions 3.5 and 3.6 in \cite{GGSp}, the following is true: $N$ is orientable and there exists a $T^1_2$-action on the normal bundle of $N$ obtained by orthogonally rotating the fibers. Using this action,  they define $\hat{M} \cong(E \times S^3 \times S^3)\times_{T^1_1} D^2\cup (E \times S^3 \times S^3)\times_{T^1_2}  D^2$ and show in \cite{GGSp} that one obtains the following fibrations:

$$S^3 \times S^3 \rightarrow \hat{M} \rightarrow M,$$

and 
$$S^3  \rightarrow \hat{M}  \rightarrow  E \times_{T^2} (S^3 \times S^3).$$ 

Since $M$ is simply-connected, it follows that $\hat{M}$ is simply connected and so $\pi_1(E \times_{T^2} (S^3 \times S^3)) = 0$ and we get the following
three exact sequences: 
$$ 0\rightarrow \pi_2(E \times S^3 \times S^3)\rightarrow \pi_2(E \times_{T^2} (S^3 \times S^3))\rightarrow \pi_1(T^2 )\xrightarrow{j} \pi_1(E \times S^3 \times S^3)\rightarrow 0,$$

$$0\rightarrow \pi_2(E)\rightarrow \pi_2(F) \rightarrow \pi_1(T^1_1) \xrightarrow{i_1} \pi_1(E) \xrightarrow{p_1} \pi_1(F) \rightarrow 0,$$

$$0\rightarrow \pi_2(E) \rightarrow \pi_2(N) \rightarrow \pi_1(T^1_2) \xrightarrow{i_2} \pi_1(E) \xrightarrow{p_2}  \pi_1(N) \rightarrow 0.$$
Let $$k= i_1\oplus i_2 : \pi_1(T^1_1) \oplus \pi_1(T^1_2) \rightarrow \pi_1(E).$$
Since $j$ is surjective,  $k$ is surjective. Then
$\pi_1(T^1_2) \rightarrow \pi_1(E)/\textrm{im}(i_1) \cong \pi_1(F)$ is surjective and hence  $\pi_1(F)$ is cyclic. A similar argument gives us that $\pi_1(N)$ is cyclic.
\end{proof}

\section{Proof of Main Theorem}\label{4}

In this section we present the proof of Main Theorem. 
We first recall the following lemma from \cite{GGS2}.
\begin{lemma}\label{l:nofreeoralmostfree}  \cite{GGS2}  Let $T^n$ act on $M^{n+3}$, a closed, simply-connected smooth manifold. Then $T^n$ cannot act freely or almost freely; that is, some circle subgroup has non-trivial fixed point set.
\end{lemma}
Note that 
by Lemma \ref{l:nofreeoralmostfree}, a $T^3$-action on a closed, simply-connected $M^6$ must have circle isotropy. Therefore, we may break the proof of the Main Theorem into 
three cases, depending on the rank of the largest isotropy subgroup, which will be either $1$, $2$ or $3$.
Theorem \ref{Wiemeler}  gives us an equivariant diffeomorphism classification of those manifolds with $T^3$ isotropy.   Thus, we have proven Part (1)  of the following theorem.

\begin{theorem}\label{isotropy} Let $M^6$ be a closed, simply-connected, non-negatively curved Riemannian $6$-manifold admitting an isometric, effective $T^3$-action.  Then the  action has singular isotropy of rank $1$, $2$ or $3$ and the following hold.
\begin{enumerate}
\item If the rank of the largest singular isotropy subgroup is equal to $3$, then $M^6$ is equivariantly diffeomorphic to a torus manifold, that is, it is equivariantly diffeomorphic to one of 
$S^6$, $\ccc P^3$, $(S^3\times S^4)/T^1$, $(S^3\times S^5)/T^2$ or $(S^3\times S^3\times S^3)/T^3$.
\item If the rank of the largest singular isotropy subgroup is less than or equal to $2$, then $M^6$ is equivariantly diffeomorphic to $S^3\times S^3$.
\end{enumerate}
\end{theorem} 

It remains to prove Part (2) of Theorem \ref{isotropy}.   We break the proof into two cases: Case (A), where the action is $T^1$-fixed point homogeneous for some $T^1\subset T^3$ and Case (B), where 
no circle subgroup acts fixed point homogeneously.

\subsection{Proof of Case (A) of Part (2) of Theorem  \ref{isotropy}}
\label{ss:1}

We have two further subcases to consider: Case (A1), where the action admits only $T^1$ isotropy and Case (A2), where the actions admits $T^2$ isotropy.

We first consider Case (A1), where some circle acts fixed point homogeneously and the induced $T^2$-action on the codimension two fixed point set is either free or almost free.
We will prove the following theorem.

\begin{theorem}\label{fph} Let $T^3$ act isometrically and effectively on $M^6$, a closed, simply-connected, non-negatively curved Riemannian manifold. Suppose that 
the action is $S^1$-fixed point homogeneous and that the largest isotropy subgroup of the $T^3$-action is of rank one. Then $M$ is equivariantly diffeomorphic to $S^3\times S^3$ with a smooth $T^3$-action.
\end{theorem}

The strategy for the proof of Theorem \ref{fph} will be to show that $M^6$ decomposes as a union of two disk bundles, each a $2$-disk bundle over a $4$-manifold. One can then show that $M^6$ has the homology groups of $S^3\times S^3$ and by Corollary \ref{Wall}, we then obtain a diffeomorphism classification. In order to show the equivariant diffeomorphism, we will need to prove that the $4$-manifolds are equivariantly diffeomorphic to $S^1\times S^3$ and use Lemma \ref{equivariance}.

We begin by establishing some notation. Let $F$ be the fixed point set component of the circle action of maximal dimension on $M^6$ and let $N$ be as in Theorem \ref{Spindeler} such that $M^6$ is given as 
\beq\label{disc}
M  = D(F) \cup_E D(N),
\eeq
where $E$ is the common boundary of the two disk bundles. 
Observe that $F$ is a closed, orientable, non-negatively curved $4$-dimensional submanifold of $M^6$, admitting an isometric $T^2$-action.  Among other things, we will show in Proposition \ref{p:fix}
that under these hypotheses, $N$ is also $4$-dimensional. 

\begin{remark} For the remainder of this subsection, we will always assume that there is a $T^3$ isometric and effective action on $M^6$,  a closed, simply-connected, non-negatively curved Riemannian manifold, such that 
the action is $S^1$-fixed point homogeneous and the largest isotropy subgroup of the $T^3$-action is of rank one. As such, we will omit the statement of these hypotheses in what follows. 
\end{remark}

The following proposition shows that the topology of both $F$ and $N$ is strongly restricted when $M^6$ is $S^1$-fixed point homogeneous. 

\begin{proposition}\label{p:fix} Let  $M'$ denote either $F$ or $N$.  Then the following are true: 
\begin{enumerate}
\item $\pi_1(M')\cong \zzz$;
\item $\chi(M')=0$;
\item $M'$ is orientable; and 
\item $\dim(M')=4$.
\end{enumerate}

\end{proposition}

\begin{proof}
We will first prove the proposition holds for $M'=F$.
If we assume that $\chi(F)\neq 0$, then the induced $T^2$-action on $F$ would have non-empty fixed point set and thus there is a point in $M^6$ fixed by $T^3$, contrary to our hypothesis that the isotropy subgroups have rank at most $1$. Thus, $\chi(F)=0$. 

It follows from Lemma \ref{l:pi1}  that  $\rk(H_1(F))\leq 1$. Suppose then that $\rk(H_1(F))= 0$, to obtain a contradiction.
$F$ is orientable, since it is a fixed point set of a circle action and $M^6$ is orientable. Therefore $\chi(F)$ is strictly positive, a contradiction.
Thus $\rk(H_1(F))=1$.  By Theorem \ref{cyclic}, it follows that $\pi_1(F)\cong \zzz$.

We will now show that the proposition holds for $M'=N$. We will first show that $\dim(N)=4$.
Since $M^6$ decomposes as a union of disk bundles over $F$ and $N$, respectively, and $M^6$ is simply-connected, from the Mayer Vietoris sequence of the triple $(M, F, N)$, we have the following long exact sequence:
\beq\tag{4.2}\label{1}
\cdots \rightarrow H_1(E)\rightarrow H_1(F)\oplus H_1(N)\rightarrow 0.
\eeq
 Now assume $N$ is not $4$-dimensional. Note first that by Lemma \ref{codim} $\textrm{codim}(N) \geq 3$. Moreover, by Theorem \ref{VanKampen}, $\pi_1(E)\cong \pi_1(N)$, hence 
 $H_1(E)\cong H_1(N)$.  This combined with the fact that Part (1) of the proposition holds for $F$ gives a contradiction to the fact that the map in Display (\ref{1}) is onto. Hence $N$ is $4$-dimensional.  
 
 Note that by Lemma \ref{equivariance}, $N$ is $T^3$-invariant. We have two cases to consider: Case (1), $N$ is not fixed by any circle subgroup of $T^3$ and Case (2), $N$ is fixed by some circle subgroup of $T^3$.  Since $M$ is simply-connected, the  Double Soul Theorem \ref{DST} 
implies that $N$ cannot be fixed by the same circle that fixes $F$, so we may apply Theorem \ref{cyclic} to show that $\pi_1(N)\cong \zzz$. 
The remainder of the results now follow as they did  for $F$.
 
It remains to consider Case (1), that is 
where $N$ is not fixed by any circle subgroup. Then, since it is invariant under the $T^3$-action, it follows that it is a cohomogeneity one submanifold and hence diffeomorphic 
to $S^1\times M^3$ (see Pak \cite{Pak} and Parker \cite{Pa}). Recall by Lemma \ref{l:pi1} that $\rk(H_1(N))\leq 1$, so by the K\"unneth formula, it follows that  $H_1(M^3)$ is finite. Hence $M^3$ is one of  $S^3$ or $L_{p, q}$  (see Mostert \cite{M} and Neumann \cite{N}). 
In particular, $N$ is an orientable submanifold with $\chi(N)=0$ and we may apply Theorem \ref{cyclic} once again to show that $\pi_1(N)\cong \zzz$. 
\end{proof}

We can now prove the following proposition, which tells us  that $M^6$ has the same homology groups as $S^3\times S^3$.

\begin{proposition}\label{homology} The homology groups of $M^6$ are isomorphic to those
of $S^3\times S^3$, that is, 
\begin{equation*}  
 H_i(M^6) \cong H_i(S^3 \times S^3)  \,\,\text{for all} \,\,i.   
 \end{equation*}
\end{proposition}

\begin{proof}
Consider the Mayer Vietoris sequence of the disk bundle decomposition for $M^6$. Using Poincar\'e Duality and the Universal Coefficient Theorem,  one immediately concludes that $H_2(M^6)\cong H_4(M^6)=0$ and that $H_3(M^6)$ has no torsion. So the only unknown homology group is $H_3(M^6)$. 
Using the Gysin sequence, we see that $\rk(H_3(E))\leq 1$. Further, using the Universal Coefficient Theorem, it follows that $H_3(E)$ has no torsion, thus $H_3(E)$ is either trivial or $\zzz$, and that $H_3(E) \cong \text{Hom}(H_2(E);\zzz)$.  We then have the following exact sequence from the Mayer Vietoris sequence:
$$0\rightarrow H_3(E)\rightarrow \zzz^2\rightarrow H_3(M^6)\rightarrow H_2(E)\rightarrow 0.$$
Now, considering the two possibilities for $H_3(E)$, we find that in both cases $H_3(M^6)=\zzz^2$.

\end{proof}

Combining the result of Proposition \ref{homology} with the fact that $\omega_2=0$, it follows by Corollary \ref{Wall} that $M^6$ is diffeomorphic to $S^3\times S^3$. 
We are now in a position to prove Theorem \ref{fph}.
\begin{proof}[Proof of Theorem \ref{fph}]
 Since both $N$ and $F$ are closed, orientable $4$-manifolds, using Poincar\'e duality, the Universal Coefficient theorem and the fact that $\chi(N)=\chi(F)=0$, it follows that $\beta_2=0$.  Using the fact that $\pi_1(M^4) \cong H_1(M^4) \cong \zzz$, we obtain that $H_i(M^4) \cong H_i(S^1\times S^3)$ for all $i$.
By  the classification work of Orlik and Raymond \cite{OR3}, it follows that both $N$ and $F$ are $T^2$-equivariantly diffeomorphic to $S^1\times S^3$. 
Recall that circle bundles over a base $B$ are classified by their Euler class $e\in H^2(B)$. Since $E$ is a  circle bundle over $S^1\times S^3$, it is therefore a trivial bundle and hence $E=T^2\times S^3$.
 By the classification work of \cite{OR3}, it follows that $E$ is $T^3$-equivariantly diffeomorphic to $T^2\times S^3$.
 
We now have by Lemma \ref{equivariance} that $M^6$ is $T^3$-equivariantly diffeomorphic to 
$$D^2(S^1\times S^3)\cup_{T^2\times S^3} D^2(S^1\times S^3) = S^3\times S^3.$$
We have thus shown that $M^6$ is $T^3$-equivariantly diffeomorphic to $S^3\times S^3$ with the smooth action giving the corresponding disk bundle decomposition.
\end{proof}

We now proceed to prove Case (A2). 
We will prove the following theorem.
\begin{theorem}\label{CaseB} Let $T^3$ act on $M^6$, a $6$-dimensional, closed, simply-connected, non-negatively curved Riemannian manifold. Suppose that the action admits only $T^2$ isotropy. Then $M^6$ is equivariantly diffeomorphic to $S^3\times S^3$ with a linear $T^3$-action.

\end{theorem}
For Case (A2), we note first that since there is $T^2$ isotropy, the smallest possible orbit is $T^1$ and we have the following nesting 
$$T^1\subset F^2\subset F^4\subset M^6,$$
where $F^2$ is fixed by $T^2$ and $F^4$ is fixed by a circle subgroup of $T^2$. In particular, since $F^4$ admits an induced $T^2$-action itself, but has no fixed points of this action, it is clear that $\chi(F^4)=0$.  Hence $\rk(H_1(F^4))=1$ by Lemma \ref{l:pi1}, and applying Theorem \ref{cyclic}, we obtain the following result.

\begin{lemma}\label{lemmaB} Let $F^4$ be a $4$-dimensional component of the fixed point set of some circle subgroup of the $T^3$ action on $M^6$. Suppose that $F^4$ admits a $T^2$-action with only circle isotropy. Then $\pi_1(F^4)\cong\zzz$.
\end{lemma}

 The proof of Proposition \ref{p:fix} is easily adapted to show that  the submanifold $N$ of the disk bundle decomposition of $M^6$ in Display  \ref{disc} must be $4$-dimensional and $\pi_1(N))\cong \zzz$. 
 We may then proceed as in the proof of  Theorem \ref{fph} to show that $M^6$ is equivariantly diffeomorphic to $S^3\times S^3$. In order to complete the  proof of Theorem \ref{CaseB}, it suffices to show that it is equivariantly diffeomorphic to $S^3\times S^3$ with a linear $T^3$ action. However, this follows from work of McGavran and Oh (see Section 2 of \cite{McGO}).

\subsection{Proof of Case B of Part 2 of Theorem \ref{isotropy}} \label{ss:2}

We now consider the case where there is only isolated circle isotropy, that is where the rank of the isotropy subgroups is at most one and the action is not $S^1$-fixed point homogeneous. The goal of this subsection is to prove the following theorem.

\begin{theorem}\label{ss2main} Let $T^3$ act on $M^6$, a $6$-dimensional, closed, simply-connected, non-negatively curved Riemannian manifold. Suppose that the largest isotropy subgroup of the $T^3$-action is of rank two. Then $M^6$ is equivariantly diffeomorphic to $S^3\times S^3$ with a linear $T^3$-action.
\end{theorem}

The argument is a straightforward generalization of the finite isotropy case for isometric $T^2$-actions on closed, simply-connected, non-negatively curved $5$-manifolds with only isolated circle orbits that appears in \cite{GGS2}. We include it here for the sake of completeness.

First recall from Corollary 4.7 of Chapter IV of Bredon \cite{Br},  that the quotient space, $M^*$, of a cohomogeneity three $G$-action on a compact, simply-connected manifold with connected orbits
is a simply-connected $3$-manifold with or without boundary. 
Note that when there is only isolated circle isotropy for a cohomogeneity three torus action, the quotient space will not have boundary and thus, by the resolution of the Poincar\'e conjecture (see Perelman \cite{P1, P2, P3}), we have that $M^*=S^3$.

We first recall the following result from \cite{GGS2} (see Proposition 4.5 and Corollary 4.6), which gives us a lower bound for the number of isolated singular orbits of the action.
 \begin{proposition}\label{4.5}\cite{GGS2} Let $T^n$ act on $M^{n+3}$, a simply-connected, smooth manifold. Suppose that $M^*$ is homeomorphic to $S^3$ and that the rank of the largest isotropy subgroup is equal to one. Then there are at least $n+1$ isolated singular orbits $T^{n-1}$.
 \end{proposition}

 The non-negative curvature hypothesis gives us an upper bound on the number of isolated $T^2$ orbits. 
The  following lemma from \cite{GW} is crucial:
\begin{lemma}\label{l:A4}\cite{GW} A three dimensional non-negatively curved Alexandrov space $X^3$ has at most four points for which the space of directions is not larger than $S^2(1/2)$.
\end{lemma}

Proposition 4.8 in \cite{GGS2} 
shows that if there is finite isotropy, it must be $\zzz_2\oplus \zzz_2$ or $\zzz_k$ and that in the latter case, those exceptional orbits are not isolated.
 Combining Proposition \ref{4.5} and Lemma \ref{l:A4} it follows that there are exactly $4$ isolated $T^2$ orbits. This result combined with the proof of Proposition 5.8 in \cite{GGS2} then tells us that $\zzz_2\oplus \zzz_2$ isotropy cannot occur.

We may summarize our results as follows. 

\begin{proposition}
Let $T^3$ act isometrically and effectively on $M^6$, a $6$-dimensional, closed, simply-connected Riemannian manifold as in Theorem \ref{ss2main}. Suppose that $M^6/T^3=M^*=S^3$. Then there are exactly $4$ isolated $T^2$ orbits 
and if there is finite isotropy, then it must be cyclic and the corresponding orbits are not isolated.
\end{proposition}

We consider first the case where there is no finite isotropy. We have the following result.
\begin{proposition} Let $T^3$ act isometrically and effectively on $M^6$, a $6$-dimensional, closed, simply-connected Riemannian manifold.  Suppose that $M^6/T^3=M^*=S^3$.  If there is no finite isotropy, then $M^6$ is diffeomorphic to $S^3\times S^3$.
\end{proposition}
\begin{proof} 
The proof of Proposition \ref{4.5} shows that $\pi_2(M^6)=0$. By the Hurewicz isomorphism, it follows that $M^6$ only has homology in dimension $3$ and by the Universal Coefficients there is no torsion.
Since the fixed point set of the $T^3$-action is empty by hypothesis, it follows that $\chi(M^6)$=0. This tells us that $b_3(M^6)=2$ and thus $M^6$ has the homology groups of $S^3\times S^3$, so by Corollary \ref{Wall}, it follows that $M^6$ is diffeomorphic to $S^3\times S^3$.
\end{proof}
 
We now consider the case where  the $T^3$-action on $M^6$ has non-trivial finite isotropy.  There are just
five admissible graphs corresponding to this case (see Figure \ref{F:graphs}).

\begin{figure}
\psfrag{1}{(1)}
\psfrag{2}{(2)}
\psfrag{3}{(3)}
\psfrag{4}{(4)}
\psfrag{5}{(5)}
\centering
\includegraphics[scale=0.5]{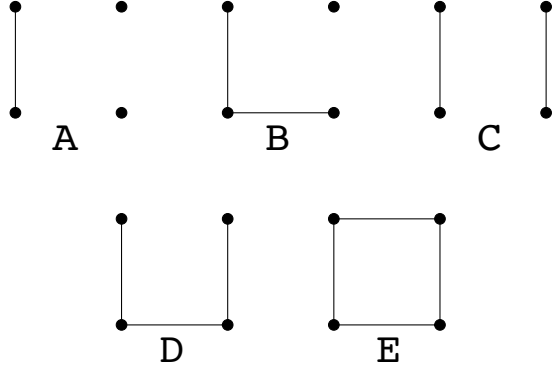}
\caption{Possible weighted graphs when there is finite cyclic isotropy.}
\label{F:graphs}
\end{figure}
In the special case where the singular set in the orbit space contains a circle we have the following result which follows directly from work of  \cite{GW} and its generalization in \cite{GGS2}.

\begin{theorem} \label{T:unknotted_circle}
Let $M^6$ be a closed, simply-connected, non-negatively curved $6$-manifold with an isometric $T^3$-action and orbit space  $M^*\simeq S^3$. If the singular set in the orbit space $M^*$ contains a circle $K^1$, then the following hold:
\begin{itemize}
\item[(1)]  The circle $K^1$ is the only circle in the singular set in $M^*$.
\item[(2)]  $K^1$ comprises all of the singular set, i.e., $M^*\setminus K^1$ is smooth.
\item[(3)]  The circle $K^1$ is unknotted in $M^*$.
 \end{itemize}
\end{theorem}
 We will now show in all cases where we have a circle that we may decompose the 
manifold as a union of disk bundles, where at least one of the disk bundles is over one arc of the circle.

\begin{proposition} Let $T^3$ act on $M^6$ isometrically and effectively and suppose that $M^*=S^3$ and there is finite isotropy. Suppose further that the singular set in $S^3$ corresponds to graph $E$ in figure \ref{F:graphs}. Then we may decompose $M^6$ as a union of disk bundles over two disjoint $4$-dimensional submanifolds fixed by finite isotropy (although not necessarily the same group). 
\end{proposition}

The proof of this proposition is exactly the same as in \cite{GGS2} (see the proofs of Proposition 6.9 and Proposition 6.7(2)).
\begin{figure}
\centering
\includegraphics[scale=0.5]{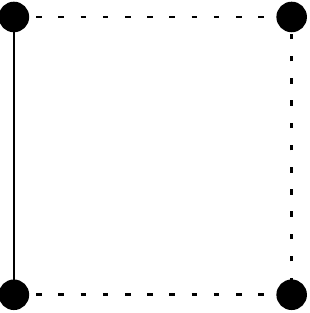}
\caption{How to complete a weighted graph with edges corresponding to principal orbits to obtain a circle: the solid edge corresponds to orbits with finite cyclic isotropy, while the dotted edges correspond to principal orbits.}
\label{F:4_vertices_completed_arc}
\end{figure}
For graphs (A) through (D),  we may complete the weighted graph by joining disjoint isolated circle orbits or arcs via edges corresponding to shortest geodesics consisting of regular points in the orbit space. In this way we obtain a graph that is an unknotted circle (see figure~\ref{F:4_vertices_completed_arc}) and now for all the possible graphs we may decompose $M^6$ as the union of two disk bundles over 
the $4$-dimensional manifolds that correspond to opposite arcs of the circle. These $4$-dimensional manifolds are invariant under the $T^3$-action and via the classification of torus actions of cohomogeneity one (see \cite{Pa, Pak}), it follows that they are $T^1\times M^3$, where $M^3$ is an orientable, cohomogeneity one manifold equal to one of $S^3$, $\lpq$, $S^2\times S^1$ by \cite{M} and  \cite{N}.  By Lemma  \ref{l:pi1}, it follows that the $4$-dimensional manifold 
may be one of $S^1\times S^3$ or $S^1\times \lpq$. 

As in Case (1) of Part 2 of Theorem \ref{isotropy}, 
analyzing the Mayer-Vietoris sequence of the decomposition it is immediate that the $4$-dimensional manifolds
corresponding to opposite arcs for all the graphs must be $S^1\times S^3$ and $M^6$ has the homology groups of $S^3\times S^3$.  Applying Corollary \ref{Wall}, it follows that 
$M^6$ is diffeomorphic to $S^3\times S^3$.

It remains to show that the classification is up to equivariant diffeomorphism in both cases. The argument  in the proof of Proposition \ref{T:unknotted_circle} (see \cite{GW} and \cite{GGS2}) uses the construction of a vector field $V^*$ on $M^*$. We construct $V^*$  so that the flow lines emanating from each point of one edge will meet at a point of the other edge to form a $2$-sphere, unless the points are vertices of the rectangle, in which case there is only one flow line. Moreover, there is an $S^1$-action on $M^*=S^3$ preserving these spheres with orbit space a $2$-dimensional rectangle. This action clearly lifts to an action on $M$ whose orbits near the two $4$-dimensional submanifolds are the normal circles in a tubular neighborhood. It follows that this lift commutes with the given isometric $T^3$-action on $M^6$. Thus the $T^3$-action on $M^6$ extends to a smooth $T^4$-action.

\vspace{5mm}

The following theorem will allow us to apply Theorem 2.5 in Oh \cite{Oh1} which states the following.  If the matrix of the circle isotropy subgroups of a $T^4$-action on $M^6$ has determinant $\pm1$, then $M^6$ is equivariantly diffeomorphic to $S^3\times S^3$.

\begin{theorem}\label{cohomk} Let $T^{n-k}$ act effectively on $M^n$ such that $M^n/T^{n-k}=D^k$, $k\geq 2$. Further assume that all singular isotropy is connected, all singular orbits correspond to boundary points and that there are no exceptional orbits.  Then $M^n$ is simply-connected if and only if
there are $(n-k)$ distinct circle isotropy groups whose matrix has determinant $\pm1$.
\end{theorem}

\begin{proof}   First assume that $M^n$ is simply connected.  
 Corollary 2.9 in \cite{ES}, which generalizes a result in Kim, McGavran and Pak \cite{KMP}, says that with the above hypotheses the isotropy subgroups of the $T^{n-k}$-action generate $T^{n-k}$, and there are at least $n-k$ distinct  circle isotropy subgroups. Let $\Delta$ be the matrix of the $(n-k)$ distinct circle isotropy groups.  It is shown that  the $n-k$ isotropy subgroups of the $T^{n-k}$-action generate $T^{n-k}$ if and only $\det(\Delta)=\pm 1$ in Lemma 1.4 in Oh \cite{Oh3}.

The converse is proven for $k=2$ in Corollary 1.2 in Oh \cite{Oh2}.  We can generalize the result to  $k\geq 2$ by observing that the proof 
 only requires that the regular part of the manifold be $\mathring{D}^k\times T^{n-k}$.  Hence we see that if $\det(\Delta)=\pm 1$, then $M^n$ is simply-connected.

\end{proof} 

\begin{remark} Note that Theorem \ref{cohomk} is optimal, since for $k=1$ there are cohomogeneity one $T^1$-actions on $\rrr P^2$ as well as cohomogeneity one $T^2$-actions on $L(p, 1)$ such that $\det(\Delta)=\pm 1$.
\end{remark}

Finally, we note that by work of \cite{GGK}, a smooth $T^4$-action on $M^6$, a closed, simply-connected, non-negatively curved Riemannian manifold corresponds to an isometric $T^4$-action and hence by Theorem \ref{ES}, it follows that $M^6$ is equivariantly diffeomorphic to $S^3\times S^3$ with a $T^3$-subaction of a linear $T^4$-action.
This now completes the proof of Case (B) of Part $2$ of Theorem \ref{isotropy}.



\begin{thebibliography}{999}


\bibitem{Br} G. Bredon,  {\em Introduction to compact transformation groups}, Academic Press 48 (1972). 
 
 \bibitem {DV} J. DeVito, {\em The classification of compact simply connected biquotients in dimension $6$ and $7$}, Math. Ann.,  DOI 10.1007/s00208-016-1460-8 (2016).
 

\bibitem{ES} C. Escher and C. Searle,  {\em Torus actions, maximality and non-negative curvature}, arXiv preprint,  arXiv:1506.08685v3[math.DG] 15 Nov 2017 .


\bibitem {FR}  F. Fang and X. Rong, {\em Homeomorphism Classification of Positively Curved Manifolds with Almost Maximal Symmetry Rank}, Math. Ann.  {\bf 332}  no. 1 (2005) 81--101.

\bibitem{GGK} F. Galaz-Garc\'ia and M. Kerin, {\em Cohomogeneity two torus actions on non-negatively curved manifolds of low dimension}, Math. Zeitschrift {\bf 276} (1-2) (2014), 133--152.

\bibitem{GGS1} F. Galaz-Garc\'ia  and C. Searle,   {\em Low-dimensional manifolds with non-negative curvature and maximal symmetry rank}, Proc. Amer. Math. Soc. {\bf 139} (2011) 2559--2564. 

\bibitem{GGS2} F. Galaz-Garc\'ia and C. Searle,   {\em Nonnegatively curved $5$-manifolds with almost maximal symmetry rank}, Geom. Topol.  {\bf 18} (2014), no. 3, 1397--1435. 

\bibitem{GGSp}  F. Galaz-Garc\'ia and W. Spindeler {\em Nonnegatively curved fixed point homogeneous $5$-manifolds}, Ann. Global Anal. Geom., {\bf 41}, (2012), no. 2, 253--263.  	
	
\bibitem{GS} K. Grove and C. Searle,  {\em Positively curved manifolds with maximal symmetry rank}, Jour. of Pure and Appl. Alg. {\bf 91}  (1994) 137--142.

\bibitem{GW} K. Grove and B. Wilking, {\em A knot characterization and $1$-connected nonnegatively curved 4- manifolds with 
circle symmetry}, Geom. Topol. {\bf 18} no. 5 (2014), 3091--3110.
 
\bibitem{I}  H. Ishida, {\em  Complex manifolds with maximal torus actions},  Journal f\"ur die reine und angewandte Mathematik (Crelles Journal), DOI: https://doi.org/10.1515/crelle-2016-0023 (2016).

\bibitem{J} P. E. Jupp, {\em Classification of certain $6$-manifolds},  Mathematical Proceedings of the Cambridge Philosophical Society {\bf 73} no. 2 (1973) 293--300.

\bibitem{KMP}   S. Kim, D. McGavran and J. Pak, {\em Torus group actions in simply connected manifolds}, Pacific J. Math., {\bf 53}  no.2 (1976), 435--444. 

\bibitem {K} B. Kleiner,  PhD thesis, U.C. Berkeley (1989).

\bibitem{Ku} S. Kuroki, {\em An Orlik-Raymond type classification of simply-connected six-dimensional torus manifolds with vanishing odd degree cohomology},  Pac. J. Math {\bf 280} no. 1 (2016), 89--114.

\bibitem{McGO} D. McGavran, H. S. Oh, {\em Torus actions on $5$- and $6$-manifolds}, Indiana Univ. Math. J., {\bf 31}, no. 3, (1982), 363--376.

\bibitem{M}  P. S. Mostert, {\em On a compact Lie group acting on a manifold}, Ann. of Math., {\bf 65}, no. 2,  (1957), 447-455.

\bibitem{N} W. D. Neumann, {\em 3-dimensional $G$-manifolds with 2-dimensional orbits}, 1968 Proc. Conf. on Transformation Groups (New Orleans, La., 1967), Springer, New York, 220--222.

\bibitem{Oh1} H. S. Oh, {\em $6$-dimensional manifolds with effective $T^4$ actions}, Topology Appl. 13 (1982), no. 2, 137--154.

\bibitem{Oh2} H. S. Oh, {Toral actions on $5$- and $6$-dimensional manifolds}, PhD thesis, U. Michigan (1980).

\bibitem{Oh3} H. S. Oh, {\em Toral actions on $5$-manifolds}, Transactions of the AMS {\bf 278} no. 1 (1983) 233--252.

\bibitem{OR3} P. Orlik, F. Raymond, {\em Actions of the torus on $4$-manifolds, II}, Topology, {\bf 13} (1974), 89--112.

\bibitem{Pak} J. Pak, {\em Actions of torus $T^n$ on $(n+1)$-manifolds $M^{n+1}$.} Pacific J. Math., {\bf 44}, no. 2 (1973), 671--674.

\bibitem{Pa} J. Parker, {\em 4-dimensional $G$-manifolds with 3-dimensional orbit}, Pacific J. Math, {\bf 125}, no. 1 (1986), 187--204.

\bibitem{P1} G. Perelman, ``The entropy formula for the Ricci Flow and its geometric applications",  arXiv:math.DG/0211159, 
11 Nov 2002. 

\bibitem{P2} G. Perelman, ``Ricci Flow with surgery on three-manifolds", arXiv:math.DG/0303109, 10 Mar 2003. 

\bibitem{P3} G. Perelman, ``Finite extinction time for the solutions to the Ricci flow on certain three-manifolds",  arXiv:math.DG/0307245, 17 Jul 2003. 

 \bibitem{Ro} X. Rong, {\em Positively curved manifolds with almost maximal symmetry rank}, Geom. Ded. {\bf 95} (2002) 157--182.
 
\bibitem{SY} C. Searle and  D. Yang, {\em On the topology of non-negatively curved simply-connected 4-manifolds with continuous symmetry}, Duke Math. J. {\bf 74} no. 2 (1994) 547--556.

\bibitem{Spi} W. Spindeler, {\em $S^1$-actions on $4$-manifolds and fixed point homogeneous manifolds of nonnegative curvature}, PhD Thesis,  Westf\"alische Wilhelms-Universit\"at M\"unster (2014).

\bibitem{Wa}   C. T. C. Wall, {\em Classification problems in differential topology - IV},  Topology {\bf 6}  (1967)  273--296. 

 \bibitem{Wie} M. Wiemeler, {\em Torus manifolds and non-negative curvature}, J. London Math. Soc. (2015) {\bf 91} no. 3, 667-692.
 
 \bibitem{Wi1} B. Wilking,  {\em Torus actions on manifolds of positive sectional curvature},  Acta Math., {\bf 191}  no. 2  (2003) 259--297. 

\bibitem{Z1} A. V. Zhubr, {\em A decomposition theorem for simply connected $6$-manifolds}, LOMI seminar notes {\bf 36} (1973)  40--49 (Russian).

\bibitem{Z2} A. V. Zhubr, {\em Classification of simply connected six-dimensional spinor manifolds}, (English) Math. USSR, Izv. {\bf 9} (1975), (1976), 793--812.

\bibitem{Z3} A. V. Zhubr, {\em Closed simply connected six-dimensional manifolds: proofs of classification theorems}, Algebra i Analiz {\bf 12} (2000), no.4, 126--230.
 
 \end{thebibliography}
\end{document}